\documentclass[11pt]{amsart}
\usepackage{amsmath}
\usepackage[active]{srcltx}
\usepackage{t1enc}
\usepackage[latin2]{inputenc}
\usepackage{verbatim}
\usepackage{amsmath,amsfonts,amssymb,amsthm}
\usepackage[mathcal]{eucal}
\usepackage{enumerate}
\usepackage[centertags]{amsmath}
\usepackage{graphics}
\usepackage[active]{srcltx}

\newtheorem{theorem}{Theorem}
\newtheorem{lemma}{Lemma}

\newtheorem{corollary}{Corollary}

\numberwithin{equation}{subsection}

 \DeclareMathOperator{\mes}{mes}

\begin{document}
\author{Gy\"orgy G\'at and Ushangi Goginava}
\title[Convergence of logarithmic means]{Convergence of logarithmic means of
Multiple Walsh-Fourier series}
\address{G. G\'at, Institute of Mathematics and Computer Science, College of
Ny\'\i regyh\'aza, P.O. Box 166, Ny\'\i regyh\'aza, H-4400 Hungary }
\email{gatgy@nyf.hu}
\address{U. Goginava, Department of Mathematics, Faculty of Exact and
Natural Sciences, Iv. Javakhishvili Tbilisi State University, Chavchavadze
str. 1, Tbilisi 0128, Georgia}
\email{zazagoginava@gmail.com}
\thanks{The research of U. Goginava was supported by Shota Rustaveli
National Science Foundation grant DI/9/5-100/13 (Function spaces, weighted
inequalities for integral operators and problems of summability of Fourier
series). The research of G. G\'at was supported by project
T\'AMOP-4.2.2.A-11/1/KONV-2012-0051}
\maketitle

\begin{abstract}
The maximal Orlicz spaces such that the mixed logarithmic means of multiple
Walsh-Fourier series for the functions from these spaces converge in measure
and in norm are found.
\end{abstract}

\footnotetext{%
2010 Mathematics Subject Classification 42A24 .
}

Let $\mathbb{I}^{d}:=[0,1)^{d}$ denote a cube in the $d$-dimensional
Euclieadan space $\mathbb{R}^{d}$. The elements of $\mathbb{R}^{d}$ are
denoted by $\mathbf{x:=}\left( x_{1},...,x_{d}\right) $.

Let $D:=\left\{ 1,2,...,d\right\} ,B:=\left\{ l_{1},l_{2},...,l_{r}\right\}
,1\leq r\leq d,B\subset M, l_{k}<l_{k+1},k=1,2,...,r-1,
B^{\prime}:=D\backslash B$. For any $\mathbf{x=}\left(
x_{1},...,x_{d}\right) $ and any $B\subset D$, denote $\mathbf{x}%
_{B}:=\left( x_{l_{1}},x_{l_{2}},...,x_{l_{r}}\right) \in \mathbb{R}^{r}$.
We assume that $|B|$ is the number of elements of $B$. If $B\neq \varnothing 
$, then for any natural numbers $n$ we suppose that $n\left( B\right)
:=\left( n,n,...,n\right) \in \mathbb{R}^{|B|}.$ The notation $a\lesssim b$
in the paper stands for $a\leq cb$, where $c$ is an absolute constant. For
any $\mathbf{x}=\left( x_{1},...,x_{d}\right) $ and $\mathbf{y}=\left(
y_{1},...,y_{d}\right) $ the vector $\left( x_{1}\overset{\cdot }{+}%
y_{1},...,x_{d}\overset{\cdot }{+}y_{d}\right) $ of the space $\mathbb{R}%
^{d} $ is denoted by $\mathbf{x}\overset{\cdot }{+}\mathbf{y}$, where $%
\overset{\cdot }{+}$ denotes the dyadic addition \cite{G-E-S,S-W-S}.

Below we identify the symbols%
\begin{equation*}
\sum\limits_{i_{B}=0_{B}}^{n_{B}}\text{ \ and }\sum%
\limits_{i_{l_{1}}=0}^{n_{l_{1}}}\cdots
\sum\limits_{i_{l_{r}}=0}^{n_{l_{r}}},d\mathbf{t}_{B}\text{ and }%
dt_{l_{1}}\cdots dt_{l_{r}}.
\end{equation*}

We denote by $L_{0}(\mathbb{I}^{d})$ the Lebesgue space of functions that
are measurable and finite almost everywhere on $\mathbb{I}^{d}$. mes$\left(
A\right) $ is the Lebesgue measure of the set $A\subset \mathbb{I}^{d}$.

We denote by $L_{p}\left( \mathbb{I}^{d}\right) $ the class of all
measurable functions $f$ that are $1$-periodic with respect to all variable
and satisfy%
\begin{equation*}
\left\Vert f\right\Vert _{p}:=\left( \int\limits_{\mathbb{I}%
^{d}}|f|^{p}\right) ^{1/p}<\infty .
\end{equation*}

The $weak-L_{1}\left( \mathbb{I}^{d}\right) $ space consist of all
measurable, $1$-periodic relative to each variables functions $f$ for which%
\begin{equation*}
\left\Vert f\right\Vert _{weak-L_{1}\left( \mathbb{I}^{d}\right)
}:=\sup\limits_{\lambda }\text{mes}\left\{ \mathbf{x}\in \mathbb{I}%
^{d}:\left\vert f\left( \mathbf{x}\right) \right\vert >\lambda \right\}
<\infty .
\end{equation*}

The Rademacher system is defined by 
\begin{equation*}
r_{n}\left( x\right) =r_{0}\left( 2^{n}x\right) ,\,\,\,\,\,\,\,n\geq 1%
\mbox{
\thinspace \thinspace \thinspace \thinspace \thinspace \thinspace
and\thinspace \thinspace \thinspace \thinspace \thinspace \thinspace }x\in 
\mathbb{I}^{1}.
\end{equation*}

Let $w_{0},w_{1},...\,\,$represent the Walsh functions, i.e. $w_{0}\left(
x\right) =1,\,\,$and $\,$if$\,\,n=2^{n_{1}}+\cdots +2^{n_{r}}\,$is a
positive integer with $n_{1}>n_{2}>\cdots >n_{r}\,\geq 0\,$then 
\begin{equation*}
w_{n}\left( x\right) =r_{n_{1}}\left( x\right) \cdots r_{n_{r}}\left(
x\right) .
\end{equation*}

The Walsh-Dirichlet kernel is defined by 
\begin{equation*}
D_{n}\left( x\right) =\sum\limits_{k=0}^{n-1}w_{k}\left( x\right) .
\end{equation*}

Recall that 
\begin{equation}
D_{2^{n}}\left( x\right) =\left\{ 
\begin{array}{c}
2^{n},\mbox{if }x\in \left[ 0,1/2^{n}\right) , \\ 
0,\,\,\,\mbox{if }x\in \left[ 1/2^{n},1\right) .%
\end{array}
\right.
\end{equation}

The rectangular partial sums of d-dimensional Walsh-Fourier series are
defined as follows:

\begin{equation*}
S_{N_{D}}(f;\mathbf{x})=\sum_{j_{D}=0_{D}}^{N_{D}-1_{D}}\widehat{f}\left(
j_{1},...,j_{d}\right) \prod\limits_{i=1}^{d}w_{j_{i}}(x_{i}),
\end{equation*}%
where the number 
\begin{equation*}
\widehat{f}\left( j_{1},...,j_{d}\right) =\int\limits_{\mathbb{I}%
^{d}}f\left( \mathbf{x}\right) \prod\limits_{i=1}^{d}w_{j_{i}}(x_{i})d%
\mathbf{x},
\end{equation*}%
is said to be the $\left( j_{1},...,j_{d}\right) $th Walsh-Fourier
coefficient of the function $f.$

In the literature, it is known the notion of the Riesz's logarithmic means
of a Fourier series. The $n$-th Riesz logarithmic mean of the Fourier series
of the integrable function $f$ is defined by 
\begin{equation*}
\frac{1}{l_{n}}\sum_{k=1}^{n-1}\frac{S_{k}(f)}{k},l_{n}:=\sum_{k=1}^{n-1}%
\frac{1}{k},
\end{equation*}%
where $S_{k}(f)$ is the $k$th partial sum of its Fourier series. This
Riesz's logarithmic means with respect to the trigonometric system has been
studied by a lot of authors. We mention for instance the papers of Szász,
and Yabuta \cite{sza, ya}. This mean is discussed first by Weisz \cite{we}
and then also by Simon and Gát \cite{Sim, gat} with respect to the Walsh,
Vilenkin system.

Let $\left\{ q_{k}:k\geq 0\right\} $ be a sequence of nonnegative numbers.
The Nörlund means for the Fourier series of $f$ are defined by 
\begin{equation*}
\frac{1}{\sum_{k=1}^{n-1}q_{k}}\sum_{k=1}^{n-1}q_{k}S_{n-k}(f).
\end{equation*}%
If $q_{k}=\frac{1}{k}$, then we get the (Nörlund) logarithmic means: 
\begin{equation}
L_{n}\left( f;x\right) :=\frac{1}{l_{n}}\sum_{k=1}^{n-1}\frac{S_{n-k}(f)}{k}.
\label{L_n}
\end{equation}%
Although, it is a kind of \textquotedblleft reverse\textquotedblright\
Riesz's logarithmic means. In \cite{gg} it is proved some convergence and
divergence properties of the logarithmic means of Walsh-Fourier series of
functions in the class of continuous functions, and in the Lebesgue space $L$%
. \ 

The Nörlund logarithmic and Reisz logarithmic means of multiple Fourier
series are defined by%
\begin{equation*}
L_{n_{D}}\left( f;\mathbf{x}\right) :=\frac{1}{\prod\limits_{i\in D}l_{i}}%
\sum\limits_{i_{D}=1_{D}}^{n_{D}-1_{D}}\frac{S_{n_{D}-i_{D}}\left( f;\mathbf{%
x}\right) }{\prod\limits_{j\in D}i_{j}},
\end{equation*}%
\begin{equation*}
R_{n_{D}}\left( f;\mathbf{x}\right) :=\frac{1}{\prod\limits_{i\in D}l_{i}}%
\sum\limits_{i_{D}=0_{D}}^{n_{D}-1_{D}}\frac{S_{i_{D}}\left( f;\mathbf{x}%
\right) }{\prod\limits_{j\in D}i_{j}},
\end{equation*}
where $n_D-i_D = (n_{l_1}-i_{l_1},\dots, n_{l_r}-i_{l_r})$. It is evident
that 
\begin{equation*}
L_{n_{D}}\left( f;\mathbf{x}\right) =\int\limits_{\mathbb{I}^{d}}f\left( 
\mathbf{t}\right) F_{n_{D}}\left( \mathbf{x}\overset{\cdot }{+}\mathbf{t}%
\right) d\mathbf{t}
\end{equation*}%
and%
\begin{equation*}
R_{n_{D}}\left( f;\mathbf{x}\right) =\int\limits_{\mathbb{I}^{d}}f\left( 
\mathbf{t}\right) G_{n_{D}}\left( \mathbf{x}\overset{\cdot }{+}\mathbf{t}%
\right) d\mathbf{t,}
\end{equation*}%
where 
\begin{equation*}
F_{n_{D}}\left( \mathbf{x}\right) :=\prod\limits_{j\in D}F_{n_{j}}\left(
x_{j}\right) ,G_{n_{D}}\left( \mathbf{x}\right) :=\prod\limits_{j\in
D}G_{n_{j}}\left( x_{j}\right) ,
\end{equation*}%
\begin{equation*}
F_{n}\left( u\right) :=\frac{1}{l_{n}}\sum\limits_{i=1}^{n-1}\frac{%
D_{n-i}\left( u\right) }{i},G_{n}\left( u\right) :=\frac{1}{l_{n}}%
\sum\limits_{i=1}^{n-1}\frac{D_{i}\left( u\right) }{i}.
\end{equation*}

Let $B\subset D$. Then the mixed logarithmic means of multiple Walsh-Fourier
series are defined by 
\begin{equation*}
\left( L_{n_{B}}\circ R_{n_{B^{\prime }}}\right) \left( f;\mathbf{x}\right)
:=\frac{1}{\prod\limits_{i\in D}l_{i}}\sum\limits_{i_{D}=1_{D}}^{n_{D}-1_{D}}%
\frac{S_{n_{B}-i_{B},i_{B^{\prime }}}\left( f;\mathbf{x}\right) }{%
\prod\limits_{j\in D}i_{j}}.
\end{equation*}%
It is easy to show that%
\begin{equation*}
\left( L_{n_{B}}\circ R_{n_{B^{\prime }}}\right) \left( f;\mathbf{x}\right)
=\int\limits_{\mathbb{I}^{d}}f\left( \mathbf{t}\right) F_{n_{B}}\left( 
\mathbf{x}_{B}\overset{\cdot }{+}\mathbf{t}_{B}\right) G_{n_{B^{\prime
}}}\left( \mathbf{x}_{B^{\prime }}\overset{\cdot }{+}\mathbf{t}_{B^{\prime
}}\right) d\mathbf{t.}
\end{equation*}

Let $L_{Q}=L_{Q}(\mathbb{I}^{d})$ be the Orlicz space \cite{KR} generated by
Young function $Q$, i.e. $Q$ is convex continuous even function such that $%
Q(0)=0$ and

\begin{equation*}
\lim\limits_{u\rightarrow +\infty }\frac{Q\left( u\right) }{u}=+\infty
,\,\,\,\,\lim\limits_{u\rightarrow 0}\frac{Q\left( u\right) }{u}=0.
\end{equation*}

This space is endowed with the norm 
\begin{equation*}
\Vert f\Vert _{L_{Q}(\mathbb{I}^{d})}=\inf \{k>0:\int\limits_{\mathbb{I}%
^{d}}Q(\left\vert f\right\vert /k)\leq 1\}.
\end{equation*}

In particular, if $Q(u)=u\log ^{\beta }(1+u)$ ,$u,\beta >0$, then the
corresponding space will be denoted by $L\log ^{\beta }L(\mathbb{I}^{d})$.

In the trigonometric case the rectangular partial sums of double Fourier
series $S_{n,m}\left( f;x,y\right) $ of the function $f\in L_{p}\left( 
\mathbb{T}^{2}\right) ,1<p<\infty ,\mathbb{T}:\mathbb{=[-\pi },\mathbb{\pi )}
$ converge in $L_{p}$ norm to the function $f$, as $n\rightarrow \infty $ 
\cite{Zh}. In the case $L_{1}\left( \mathbb{T}^{2}\right) $ this result does
not hold . But for $f\in L_{1}\left( \mathbb{T}\right) $, the operator $%
S_{n}\left( f;x\right) $ are of weak type (1,1) \cite{Zy}. This estimate
implies convergence of $S_{n}\left( f;x\right) $ in measure on $\mathbb{T}$
to the function $f\in L_{1}\left( \mathbb{T}\right) $. However, for double
Fourier series this result does not hold \cite{Ge,Kon,Tk1}. \ Moreover, it
is proved that quadratical partial sums $S_{n,n}\left( f;x,y\right) $ of
double Fourier series do not converge in two-dimensional measure on $\mathbb{%
T}^{2}$ \ even for functions from Orlicz spaces wider than Orlicz space $%
L\log L\left( \mathbb{T}^{2}\right) $. On the other hand, it is well-known
that if the function $f\in L\log L\left( \mathbb{T}^{2}\right) $, then
rectangular partial sums $S_{n,m}\left( f;x,y\right) $ converge in measure
on $\mathbb{T}^{2}$.

Classical regular summation methods often improve the convergence of Fourier
seeries. For instance, the Fejér means of the double Fourier series of the
function $f\in L_{1}\left( \mathbb{T}^{2}\right) $ converge in $L_{1}\left( 
\mathbb{T}^{2}\right) $ norm to the function $f$ \cite{Zh}. These means
present the particular case of the Nörlund means.

It is well know that the method of Nörlund logarithmic means of double
Fourier series, is weaker than the Cesáro method of any positive order. In 
\cite{GGT} it is proved, that these means of double Walsh-Fourier series in
general do not converge in two-dimensional measure on $\mathbb{I}^{d}$ even
for functions from Orlicz spaces wider than Orlicz space $L\log
^{d-1}L\left( \mathbb{I}^{d}\right) $. Thus, not all classic regular
summation methods can improve the convergence in measure of double Fourier
series.

The results for summability of logarithmic means of Walsh-Fourier series can
be found in \cite{GGT2,GogJAT, GGNSMH,gg,sza,ya}.

In this paper we consider the mixed logarithmic means $\left( L_{n_{B}}\circ
R_{n_{B^{\prime }}}\right) \left( f\right) $ of rectangular partial sums
multiple Walsh-Fourier series and prove that these means are acting from
space $L\log ^{|B|}L\left( \mathbb{I}^{d}\right) $ into space \ $L_{1}\left( 
\mathbb{I}^{d}\right) $ and from space $L\log ^{|B|-1}L\left( \mathbb{I}%
^{d}\right) $ into space $weak-L_{1}\left( \mathbb{I}^{d}\right) $. Thess
facts implies the convergence of mixed logarithmic means of rectangular
partial sums of multiple Fourier series converge in norm and in measure (see
Theorem \ref{measureconv}). We also prove sharpness of these results. In
particular, the following are true.

\begin{theorem}
\label{logest1}Let $B\subset D$ and $f\in L\log ^{|B|}L\left( \mathbb{I}%
^{d}\right) $. Then%
\begin{equation*}
\left\Vert \left( L_{n_{B}}\circ R_{n_{B^{\prime }}}\right) \left( f\right)
\right\Vert _{L_{1}\left( \mathbb{I}^{d}\right) }\lesssim 1+\left\Vert
\left\vert f\right\vert \log ^{|B|}\left\vert f\right\vert \right\Vert
_{L_{1}\left( \mathbb{I}^{d}\right) }.
\end{equation*}
\end{theorem}

\begin{theorem}
\label{normconv}

Let $B\subset D$ and $f\in L\log ^{|B|}L\left( \mathbb{I}^{d}\right) $. Then 
\begin{equation*}
\left\Vert \left( L_{n_{B}}\circ R_{n_{B^{\prime }}}\right) \left( f\right)
-f\right\Vert _{L_{1}\left( \mathbb{I}^{d}\right) }\rightarrow 0\text{ \ as
\ }n_{i}\rightarrow \infty ,i\in D\text{;}
\end{equation*}
\end{theorem}

\begin{theorem}
\label{normdiv} Let $L_{Q}\left( \mathbb{I}^{d}\right) $ be an Orlicz space,
such that%
\begin{equation*}
L_{Q}\left( \mathbb{I}^{d}\right) \nsubseteqq L\log ^{|B|}L\left( \mathbb{I}%
^{d}\right) .
\end{equation*}%
Then\newline
a)%
\begin{equation*}
\sup\limits_{n}\left\Vert L_{n\left( B\right) }\circ R_{n\left( B^{\prime
}\right) } \right\Vert _{L_{Q}\left( \mathbb{I}^{d}\right) \rightarrow
L_{1}\left( \mathbb{I}^{d}\right) }=\infty ;
\end{equation*}%
\newline
\newline
b) there exists a function $f\in L_{Q}\left( \mathbb{I}^{d}\right) $ such
that $\left( L_{n\left( B\right) }\circ R_{n\left( B^{\prime }\right)
}\right) \left( f\right) $ does not converge to $f$ in $L_{1}\left( \mathbb{I%
}^{d}\right) $-norm.
\end{theorem}

Thus, the space $L\log ^{|B|}L\left( \mathbb{I}^{d}\right) $ is maximal
Orlicz space such that for each function $f$ from this space the means $%
\left( L_{n\left( B\right) }\circ R_{n\left( B^{\prime }\right) }\right)
\left( f\right) $ converge to $f$ in $L_{1}\left( \mathbb{I}^{d}\right) $%
-norm.

\begin{theorem}
\label{logest2}Let $B\subset D$ and $f\in L\log ^{|B|-1}L\left( \mathbb{I}%
^{d}\right) $. Then%
\begin{equation*}
\left\Vert \left( L_{n_{B}}\circ R_{n_{B^{\prime }}}\right) \left( f\right)
\right\Vert _{weak-L_{1}\left( \mathbb{I}^{d}\right) }\lesssim 1+\left\Vert
\left\vert f\right\vert \log ^{|B|-1}\left\vert f\right\vert \right\Vert
_{L_{1}\left( \mathbb{I}^{d}\right) }.
\end{equation*}
\end{theorem}

\begin{theorem}
\label{measureconv}

Let $B\subset D$ and $f\in L\log ^{|B|-1}L\left( \mathbb{I}^{d}\right) $.
Then 
\begin{equation*}
\left( L_{n_{B}}\circ R_{n_{B^{\prime }}}\right) \left( f\right) \rightarrow
f\text{ in measure on }\mathbb{I}^{d}\text{, \ as \ }n_{i}\rightarrow \infty
,i\in D\text{.}
\end{equation*}
\end{theorem}

\begin{theorem}
\label{baire} Let $B\subset D,|B|>1$ and $L_{Q}\left( \mathbb{I}^{d}\right) $
be an Orlicz space, such that%
\begin{equation*}
L_{Q}\left( \mathbb{I}^{d}\right) \nsubseteqq L\log ^{|B|-1}L\left( \mathbb{I%
}^{d}\right) .
\end{equation*}%
Then the set of the functions from the Orlicz space $L_{Q}\left( \mathbb{I}%
^{d}\right) $ with logarithmic means $\left( L_{n_{B}}\circ R_{n_{B^{\prime
}}}\right) \left( f\right) $ of rectangular partial sums of multiple Fourier
series convergent in measure on $\mathbb{I}^{d}$ is of first Baire category
in $L_{Q}\left( \mathbb{I}^{d}\right) .$\newline
\end{theorem}

\begin{corollary}
\label{measuredivergence}Let $B\subset D,|B|>1$ and $\varphi :[0,\infty
\lbrack \rightarrow \lbrack 0,\infty \lbrack $ be a nondecreasing function
satisfying for $x\rightarrow +\infty $ the condition

\begin{equation*}
\varphi (x)=o(x\log ^{|B|-1}x).
\end{equation*}

Then there exists the function $f\in L_{1}(\mathbb{I}^{d})$ such that

a) 
\begin{equation*}
\int\limits_{\mathbb{I}^{d}}\varphi (\left\vert f\right\vert )<\infty ;
\end{equation*}

b) logarithmic means $\left( L_{n_{B}}\circ R_{n_{B^{\prime }}}\right)
\left( f\right) $ of rectangular partial sums of multiple Fourier series of $%
f$ diverges in measure on $\mathbb{I}^{d}$.
\end{corollary}

\section{Auxiliary Results}

We apply the reasoning of \cite{G} formulated as the following proposition
in particular case.

\begin{theorem}
\label{Garsia} Let $H:L_{1}(\mathbb{I}^{d})\rightarrow L_{0}(\mathbb{I}^{d})$
be a linear continuous operator, which commutes with family of translations $%
\mathcal{E}$, i. e. $\forall E\in \mathcal{E}\quad \forall f\in L_{1}(%
\mathbb{I}^{d})\quad HEf=EHf$. Let $\Vert f\Vert _{L_{1}(\mathbb{I}^{d})}=1$
and $\lambda >1$. 
Then for any $1\leq r\in \mathbb{N}$ under condition mes$\{\mathbf{x}\in 
\mathbb{I}^{d}:|Hf|>\lambda \}\geq \frac{1}{r}$ there exist $%
E_{1},...,E_{r}\in \mathcal{E}$ and $\varepsilon _{i}=\pm 1,\quad i=1,...,r$
such that 
\begin{equation*}
\mes\{\mathbf{x}\in \mathbb{I}^{d}:|H(\sum_{i=1}^{r}\varepsilon _{i})f(E_{i}%
\mathbf{x})|>\lambda \}\geq \frac{1}{8}.
\end{equation*}
\end{theorem}

\begin{lemma}
\label{GGT}Let $\{H_{m}\}_{m=1}^{\infty }$ be a sequence of linear
continuous operators, acting from Orlicz space $L_{Q}(\mathbb{I}^{d})$ in to
the space $L_{0}(\mathbb{I}^{d})$. Suppose that there exists a sequence of
functions $\{\xi _{k}\}_{k=1}^{\infty }$ from unit ball $S_{Q}(0,1)$ of
space $L_{Q}(\mathbb{I}^{d})$, sequences of integers $\{m_{k}\}_{k=1}^{%
\infty }$ and $\{\nu _{k}\}_{k=1}^{\infty }$ increasing to infinity such that

\begin{equation*}
\varepsilon _{0}=\inf_{k}\mes\{\mathbf{x}\in \mathbb{I}^{d}:|H_{m_{k}}\xi
_{k}\left( x,y\right) |>\nu _{k}\}>0.
\end{equation*}

Then $K$ - the set of functions $f$ from space $L_{Q}(\mathbb{I}^{d})$, for
which the sequence $\{H_{m}f\}$ converges in measure to an a. e. finite
function is of first Baire category in space $L_{Q}(\mathbb{I}^{d})$.
\end{lemma}

The proof of Lemma \ref{GGT} can be found in \cite{GGT}.

\begin{lemma}
\label{GGT2}Let $L_{\Phi }\left( \mathbb{I}^{d}\right) $ be an Orlicz space
and let $\varphi :[0,\infty )\rightarrow \lbrack 0,\infty )$ be measurable
function with condition $\varphi \left( x\right) =o\left( \Phi \left(
x\right) \right) $ as $x\rightarrow \infty .$ Then there exists Orlicz space 
$L_{\omega }\left( \mathbb{I}^{d}\right) $, such that $\omega \left(
x\right) =o\left( \Phi \left( x\right) \right) \,\,$ as $x\rightarrow \infty 
$, and $\omega \left( x\right) \geq \varphi \left( x\right) $ for $x\geq
c\geq 0.$
\end{lemma}

The proof of Lemma \ref{GGT2} can be found in \cite{GGT2}.

Set $m_{*}:=\lfloor m/2\rfloor$ (the lower integer part of $m/2$) 
\begin{equation*}
\widetilde{m}:=\left\lfloor \frac{l_{p_{m_{*}}-1}}{16}-2^{15}\right\rfloor
,J_{m}:=\left[ \frac{1}{2^{m+1}},\frac{1}{2^{m+1}}+\frac{1}{2^{m+\widetilde{m%
}}}\right) ,
\end{equation*}%
\begin{equation*}
J:=\bigcup\limits_{m=m_{0}}^{\infty }J_{m},m_{0}:=\inf \left\{
m:\left\lfloor \frac{l_{p_{m_{*}}-1}}{16}-2^{15}\right\rfloor >1\right\},
\end{equation*}%
\begin{equation*}
\Omega _{n}:=\bigcup\limits_{m=n}^{2n}\left[ \frac{1}{2^{m+1}}+\frac{1}{2^{m+%
\widetilde{m}}},\frac{1}{2^{m}}\right) ,
\end{equation*}%
where $p_{n}:=2^{2n}+2^{2n-2}+\cdots +2^{0}$.

\begin{lemma}
\label{GG}For $x\in \Omega _{n}$ we have an estimation%
\begin{equation*}
\left\vert F_{p_{n}}\left( x\right) \right\vert \gtrsim \frac{1}{x}.
\end{equation*}
\end{lemma}

The proof of Lemma \ref{GG} can be found in \cite{GGT}.

\section{Proof of the Theorems\newline
}

\begin{proof}[Proof of Theorem \protect\ref{logest1}]
We apply the following particular case of the Marcinkiewicz interpolation
theorem \cite{E}. Let $T:L_{1}\left( \mathbb{I}^{1}\right) \rightarrow
L_{0}\left( \mathbb{I}^{1}\right) $ be a quasilinear operator of weak type $%
(1,1)$ and of type $\left( \alpha ,\alpha \right) $ for some $1<\alpha
<\infty $ at the same time, i. e.

\begin{eqnarray}
&&\text{mes}\left\{ x\in \mathbb{I}^{1}:\left\vert T\left( f,x\right)
\right\vert >y\right\}  \label{T1} \\
&\lesssim &\frac{1}{y}\int\limits_{\mathbb{I}^{1}}\left\vert f\left(
x\right) \right\vert dx;\,\,\,\,\forall f\in L_{1}\left( \mathbb{I}%
^{1}\right) \,\,\forall y>0;  \notag
\end{eqnarray}

and 
\begin{equation}
\,\,\,\,\,\left\Vert Tf\right\Vert _{L_{\alpha }\left( \mathbb{I}^{1}\right)
}\lesssim \,\,\left\Vert f\right\Vert _{L_{\alpha }\left( \mathbb{I}%
^{1}\right) },\,\,\,\forall f\in L_{\alpha }\left( \mathbb{I}^{1}\right) .
\end{equation}%
Then%
\begin{eqnarray}
&&\int\limits_{\mathbb{I}^{1}}\left\vert T\left( f,x\right) \right\vert \ln
^{\beta }\left\vert T\left( f,x\right) \right\vert dx\, \\
&\lesssim &\int\limits_{\mathbb{I}^{1}}\left\vert f\left( x\right)
\right\vert \ln ^{\beta +1}\left\vert f\left( x\right) \right\vert
dx+1\,,\,\,\,\forall \beta \geq 0.  \notag
\end{eqnarray}

In \cite{GGTATA} it is proved that for any $f\in L_{1}\left( \mathbb{I}%
^{1}\right) $ the operator $f\ast F_{n}$ has weak type (1,1) , i. e.%
\begin{equation}
\left\Vert f\ast F_{n}\right\Vert _{weak\_L_{1}\left( \mathbb{T}^{1}\right)
}\lesssim \left\Vert f\right\Vert _{L_{1}\left( \mathbb{T}^{1}\right) }.
\end{equation}

On the other hand, it is easy to prove that the operator $f\ast G_{n}$ has
type (1,1), i.e.%
\begin{equation}
\left\Vert f\ast G_{n}\right\Vert _{L_{1}\left( \mathbb{T}^{1}\right)
}\lesssim \left\Vert f\right\Vert _{L_{1}\left( \mathbb{T}^{1}\right) }.
\label{strong}
\end{equation}

From (\ref{T1})-(\ref{strong}) we have ($B^{\prime }:=\left\{
s_{1},s_{2},...,s_{r^{\prime }}\right\} $)%
\begin{eqnarray}
&&\left\Vert \left( L_{n_{B}}\circ R_{n_{B^{\prime }}}\right) \left(
f\right) \right\Vert _{L_{1}\left( \mathbb{I}^{d}\right) }  \label{op} \\
&=&\left\Vert \left( R_{n_{s_{1}}}\circ \cdots \circ R_{n_{s_{r^{\prime
}}}}\circ L_{n_{l_{1}}}\circ \cdots \circ L_{n_{l_{r}}}\right) \left(
f\right) \right\Vert _{L_{1}\left( \mathbb{I}^{d}\right) }  \notag \\
&\lesssim &\cdots \lesssim \left\Vert \left( L_{n_{l_{1}}}\circ \cdots \circ
L_{n_{l_{r}}}\right) \left( f\right) \right\Vert _{L_{1}\left( \mathbb{I}%
^{d}\right) }  \notag \\
&\lesssim &1+\left\Vert \left\vert L_{n_{l_{2}}}\circ \cdots \circ
L_{n_{l_{r}}}\left( f\right) \right\vert \log \left\vert L_{n_{2}}\circ
\cdots \circ L_{n_{l_{r}}}\left( f\right) \right\vert \right\Vert
_{L_{1}\left( \mathbb{I}^{d}\right) }  \notag \\
&\lesssim &\cdots \lesssim 1+\left\Vert \left\vert L_{n_{l_{r}}}\left(
f\right) \right\vert \log ^{r-1}\left\vert L_{n_{l_{r}}}\left( f\right)
\right\vert \right\Vert _{L_{1}\left( \mathbb{I}^{d}\right) }  \notag \\
&\lesssim &1+\left\Vert \left\vert f\right\vert \log ^{r}\left\vert
f\right\vert \right\Vert _{L_{1}\left( \mathbb{I}^{d}\right) }.  \notag
\end{eqnarray}

Theorem \ref{logest1} is proved.
\end{proof}

By virtue of standart arguments (see \cite{Zy}) we can see the validity of
Theorem \ref{normconv}.

\begin{proof}[Proof of Theorem \protect\ref{normdiv}]
Let 
\begin{equation*}
Q\left( 2^{2n|B|}\right) \gtrsim 2^{2n|B|}\text{ \ for \ }n>n_{0}.
\end{equation*}%
By virtue of estimate (\cite{KR}, Ch. 2)%
\begin{equation*}
\left\Vert f\right\Vert _{L_{Q}\left( \mathbb{I}^{d}\right) }\leq
1+\left\Vert Q\left( |f|\right) \right\Vert _{L_{1}\left( \mathbb{I}%
^{d}\right) }.
\end{equation*}%
We can write%
\begin{eqnarray}
&&\left\Vert L_{p_{n}\left( B\right) }\circ R_{p_{n}\left( B^{\prime
}\right) }\left( \bigotimes\limits_{i\in B}\frac{D_{2^{2n+1}}}{2}\right)
\right\Vert _{L_{1}\left( \mathbb{I}^{d}\right) }  \label{1} \\
&\leq &\left\Vert L_{p_{n}\left( B\right) }\circ R_{p_{n}\left( B^{\prime
}\right) }\right\Vert _{L_{Q}\left( \mathbb{I}^{d}\right) \rightarrow
L_{1}\left( \mathbb{I}^{d}\right) }\left\Vert \bigotimes\limits_{i\in B}%
\frac{D_{2^{2n+1}}}{2}\right\Vert _{L_{Q}\left( \mathbb{I}^{d}\right) } 
\notag \\
&\leq &\left\Vert L_{p_{n}\left( B\right) }\circ R_{p_{n}\left( B^{\prime
}\right) }\right\Vert _{L_{Q}\left( \mathbb{I}^{d}\right) \rightarrow
L_{1}\left( \mathbb{I}^{d}\right) }  \notag \\
&&\times \left( 1+\left\Vert Q\left( \bigotimes\limits_{i\in B}\frac{%
D_{2^{2n+1}}}{2}\right) \right\Vert _{L_{1}\left( \mathbb{I}^{d}\right)
}\right)  \notag \\
&\lesssim &\left\Vert L_{p_{n}\left( B\right) }\circ R_{p_{n}\left(
B^{\prime }\right) }\right\Vert _{L_{Q}\left( \mathbb{I}^{d}\right)
\rightarrow L_{1}\left( \mathbb{I}^{d}\right) }\left( 1+\frac{1}{2^{2n|B|}}%
Q\left( 2^{2n|B|}\right) \right)  \notag \\
&\lesssim &\left\Vert L_{p_{n}\left( B\right) }\circ R_{p_{n}\left(
B^{\prime }\right) }\right\Vert _{L_{Q}\left( \mathbb{I}^{d}\right)
\rightarrow L_{1}\left( \mathbb{I}^{d}\right) }\frac{Q\left(
2^{2n|B|}\right) }{2^{2n|B|}}.  \notag
\end{eqnarray}

On the other hand,%
\begin{eqnarray}
&&L_{p_{n}\left( B\right) }\circ R_{p_{n}\left( B^{\prime }\right) }\left(
\bigotimes\limits_{i\in B}\frac{D_{2^{2n+1}}}{2};\mathbf{x}\right)
\label{pw} \\
&=&\frac{1}{2^{|B|}}\int\limits_{\mathbb{I}^{|B|}}\prod\limits_{i\in
B}D_{2^{2n+1}}\left( z_{i}\right) F_{p_{n}}\left( x_{i}\overset{\cdot }{+}%
z_{i}\right) d\mathbf{z}_{B}  \notag \\
&&\times \int\limits_{\mathbb{I}^{|B^{\prime }|}}\prod\limits_{j\in
B^{\prime }}G_{p_{n}}\left( x_{j}\overset{\cdot }{+}z_{j}\right) d\mathbf{z}%
_{B^{\prime }}  \notag \\
&=&\frac{1}{2^{|B|}}\int\limits_{\mathbb{I}^{|B|}}\prod\limits_{i\in
B}D_{2^{2n+1}}\left( z_{i}\right) F_{p_{n}}\left( x_{i}\overset{\cdot }{+}%
z_{i}\right) d\mathbf{z}_{B}  \notag \\
&=&\frac{1}{2^{|B|}}\prod\limits_{i\in B}\int\limits_{\mathbb{I}%
}D_{2^{2n+1}}\left( z_{i}\right) F_{p_{n}}\left( x_{i}\overset{\cdot }{+}%
z_{i}\right) dz_{j}  \notag \\
&=&\frac{1}{2^{|B|}}\prod\limits_{i\in B}S_{2^{2n+1}}\left(
F_{p_{n}};x_{i}\right)  \notag \\
&=&\frac{1}{2^{|B|}}\prod\limits_{i\in B}F_{p_{n}}\left( x_{i}\right) . 
\notag
\end{eqnarray}%
Consequently, from Lemma \ref{GG} we get%
\begin{eqnarray}
&&\left\Vert L_{p_{n}\left( B\right) }\circ R_{p_{n}\left( B^{\prime
}\right) }\left( \bigotimes\limits_{i\in B}\frac{D_{2^{2n+1}}}{2};\mathbf{x}%
\right) \right\Vert _{L_{1}\left( \mathbb{I}^{d}\right) }  \label{low0} \\
&=&\frac{1}{2^{|B|}}\prod\limits_{i\in B}\left\Vert F_{p_{n}}\left(
x_{i}\right) \right\Vert _{L_{1}\left( \mathbb{I}^{1}\right) }\gtrsim
n^{|B|}.  \notag
\end{eqnarray}

Combining (\ref{1}) and (\ref{low0}) we obtain%
\begin{equation}
\left\Vert L_{p_{n}\left( B\right) }\circ R_{p_{n}\left( B^{\prime }\right)
}\right\Vert _{L_{Q}\left( \mathbb{I}^{d}\right) \rightarrow L_{1}\left( 
\mathbb{I}^{d}\right) }\gtrsim \frac{2^{2n|B|}n^{|B|}}{Q\left(
2^{2n|B|}\right) }.  \label{main}
\end{equation}

The fact that 
\begin{equation*}
L_{Q}\left( \mathbb{I}^{d}\right) \nsubseteqq L\log ^{|B|}L\left( \mathbb{I}%
^{d}\right)
\end{equation*}%
is equalent to the condition%
\begin{equation*}
\underset{u\rightarrow \infty }{\lim \sup }\frac{u\log ^{|B|}u}{Q\left(
u\right) }=\infty .
\end{equation*}%
Thus, there exists $\left\{ u_{k}:k\geq 1\right\} $ such that%
\begin{equation*}
\lim\limits_{k\rightarrow \infty }\frac{u_{k}\log ^{|B|}u_{k}}{Q\left(
u_{k}\right) }=\infty ,u_{k+1}>u_{k},k=1,2,...,
\end{equation*}%
and a monotonically increasing sequence of positive integers $\left\{
r_{k}:k\geq 1\right\} $ such that%
\begin{equation*}
2^{2|B|r_{k}}\leq u_{k}<2^{2|B|\left( r_{k}+1\right) }.
\end{equation*}%
Then we have%
\begin{equation*}
\frac{2^{2r_{k}|B|}r_{k}^{|B|}}{Q\left( 2^{2r_{k}|B|}\right) }\gtrsim \frac{%
u_{k}\log ^{|B|}u_{k}}{Q\left( u_{k}\right) }\rightarrow \infty .
\end{equation*}%
Thus, from (\ref{main}) we conclude that

\begin{equation*}
\sup\limits_{n}\left\Vert L_{p_{n}\left( B\right) }\circ R_{p_{n}\left(
B^{\prime }\right) } \right\Vert _{L_{Q}\left( \mathbb{I}^{d}\right)
\rightarrow L_{1}\left( \mathbb{I}^{d}\right) }=\infty .
\end{equation*}

This completes the proof of Theorem \ref{normdiv} part a). Part b) follows
immediately from part a).
\end{proof}

\begin{proof}[Proof of Theorem \protect\ref{logest2}]
Set%
\begin{equation*}
\Omega :=\left\{ \mathbf{x\in }\mathbb{I}^{d}:\left\vert \left(
L_{n_{B}}\circ R_{n_{B^{\prime }}}\right) \left( f,\mathbf{x}\right)
\right\vert >\lambda \right\} .
\end{equation*}

Then from (\ref{T1})-(\ref{op}) we have 
\begin{eqnarray*}
&&\text{mes}\left\{ \mathbf{x\in }\mathbb{I}^{d}:\left\vert \left(
L_{n_{B}}\circ R_{n_{B^{\prime }}}\right) \left( f,\mathbf{x}\right)
\right\vert >\lambda \right\} \\
&=&\int\limits_{\mathbb{I}^{d}}\mathbf{1}_{\Omega }\left( \mathbf{x}\right) d%
\mathbf{x}=\int\limits_{\mathbb{I}^{d-1}}\left( \int\limits_{\mathbb{I}}%
\mathbf{1}_{\Omega }\left( \mathbf{x}\right) dx_{l_{1}}\right) d\mathbf{x}%
_{D\backslash \{l_{1}\}} \\
&\lesssim &\frac{1}{\lambda }\left\Vert \left( L_{n_{B\backslash
\{l_{1}\}}}\circ R_{B^{\prime }}\right) \left( f\right) \right\Vert
_{L_{1}\left( \mathbb{T}^{d}\right) } \\
&\lesssim &1+\left\Vert \left\vert f\right\vert \log ^{r-1}\left\vert
f\right\vert \right\Vert _{L_{1}\left( \mathbb{T}^{d}\right) }.
\end{eqnarray*}

Theorem \ref{logest2} is proved.
\end{proof}

By virtue of standart arguments (see \cite{Zy}) we can see the validity of
Theorem \ref{measureconv}.

\begin{proof}[Proof of Theorem \protect\ref{baire}]
By Lemma \ref{GGT} the proof of Theorem \ref{normdiv} will be complete if we
show that there exists sequences of integers $\{n_{k}:k\geq 1\}$ and $\{\nu
_{k}:k\geq 1\}$ increasing to infinity, and a sequence of functions $\{\xi
_{k}:k\geq 1\}$ from the unit ball $S_{Q}\left( 0,1\right) $ of Orlicz space 
$L_{Q}\left( \mathbb{I}^{d}\right) $, such that for all $k$%
\begin{equation}
\mes\{\mathbf{x}\in \mathbb{I}^{d}:\left\vert L_{p_{n_{k}}\left( B\right)
}\circ R_{p_{n_{k}}\left( B\right) }\left( \xi _{k};\mathbf{x}\right)
\right\vert >\nu _{k}\}\geq \frac{1}{8}.  \label{est0}
\end{equation}

First, we prove that%
\begin{eqnarray}
&&\mes\left\{ \mathbf{x}\in \mathbb{I}^{d}:\left\vert L_{p_{n}\left(
B\right) }\circ R_{p_{n}\left( B^{\prime }\right) }\left(
\bigotimes\limits_{i\in B}D_{2^{2n+1}};\mathbf{x}\right) \right\vert \gtrsim
2^{n\left( 2|B|-1\right) }\right\}  \label{est1} \\
&\gtrsim &\frac{n^{|B|-1}}{2^{n\left( 2|B|-1\right) }},\quad |B|>1.  \notag
\end{eqnarray}

From (\ref{pw}) and Lemma \ref{GG} we have%
\begin{eqnarray*}
&&\left\vert L_{p_{n}\left( B\right) }\circ R_{p_{n}\left( B^{\prime
}\right) }\left( \bigotimes\limits_{i\in B}D_{2^{2n+1}};\mathbf{x}\right)
\right\vert \\
&=&\prod\limits_{i\in B}\left\vert F_{p_{n}}\left( x_{i}\right) \right\vert
\\
&\gtrsim &\prod\limits_{j\in B}\frac{1}{x_{j}},x_{j}\in \Omega _{n}, \quad
j\in B.
\end{eqnarray*}%
Consequently,%
\begin{eqnarray*}
&&\mes\left\{ \mathbf{x}\in \mathbb{I}^{d}:\left\vert L_{p_{n}\left(
B\right) }\circ R_{p_{n}\left( B^{\prime }\right) }\left(
\bigotimes\limits_{i\in B}D_{2^{2n+1}};\mathbf{x}\right) \right\vert \gtrsim
2^{n\left( 2|B|-1\right) }\right\} \\
&\geq &\mes\left\{ \mathbf{x\in }\Omega _{n}^{|B|}\times \mathbb{I}%
^{|B^{\prime }|}:\prod\limits_{j\in B}\frac{1}{x_{j}}\gtrsim 2^{n\left(
2|B|-1\right) }\right\} \\
&=&\mes\left\{ \mathbf{x}_{B}\mathbf{\in }\Omega
_{n}^{|B|}:x_{l_{1}}\lesssim \frac{1}{2^{n\left( 2|B|-1\right)
}\prod\limits_{j\in B\backslash \{l_{1}\}}x_{j}}\right\} \\
&=&\sum\limits_{m_{B}=n\left( B\right) }^{2n\left( B\right) }\mes\Biggl\{ %
x_{j}\mathbf{\in }\left[ \frac{1}{2^{m_{j}+1}}+\frac{1}{2^{m_{j}+\widetilde{m%
}_{j}}},\frac{1}{2^{m_{j}}}\right) ,j\in B: \\
&& x_{l_{1}}\lesssim \frac{1}{2^{n\left( 2|B|-1\right) }\prod\limits_{j\in
B\backslash \{l_{1}\}}x_{j}}\Biggr\} \\
&\gtrsim &\sum\limits_{m_{B\backslash \{l_{1}\}}=n\left( B\backslash
\{l_{1}\}\right) }^{2n\left( B\backslash \{l_{1}\}\right)
}\sum\limits_{m_{l_{1}}=n\left( 2|B|-1\right) -\left( m_{l_{2}}+\cdots
m_{l_{r}}\right) }^{2n}\prod\limits_{j\in B}\frac{1}{2^{m_{j}}} \\
&\gtrsim &\prod\limits_{j\in B\backslash \{l_{1}\}}\sum\limits_{m_{j}=n}^{2n}%
\frac{1}{2^{m_{j}}}\sum\limits_{m_{l_{1}}=n\left( 2|B|-1\right) -\left(
m_{l_{2}}+\cdots m_{l_{r}}\right) }^{2n}\frac{1}{2^{m_{l_{1}}}} \\
&\gtrsim &\frac{1}{2^{n\left( 2|B|-1\right) }}\prod\limits_{j\in B\backslash
\{l_{1}\}}\sum\limits_{m_{j}=n}^{2n}\frac{2^{m_{j}}}{2^{m_{j}}} \\
&\gtrsim &\frac{n^{|B|-1}}{2^{n\left( 2|B|-1\right) }},\quad |B|>1.
\end{eqnarray*}

Here we also used that for $n(2|B|-1)-(m_{l_2}+\dots +m_{l_r})\le m_{l_1}
\le 2n$ we have 
\begin{eqnarray*}
x_{l_1} \le \frac{1}{2^{m_{l_1}}} \le \frac{2^{m_{l_2}+\dots +m_{l_r}}}{%
2^{n(2|B|-1)}} \le \frac{1}{2^{n(2|B|-1)} x_{l_2}\cdots x_{l_r}}.
\end{eqnarray*}

Hence (\ref{est1}) is proved.

From the condition of the theorem we write 
\begin{equation*}
\liminf_{u\rightarrow \infty }\frac{Q(u)}{u\log ^{|B|-1}u}=0.
\end{equation*}%
Consequently, there exists a sequence of integers $\{n_{k}:k\geq 1\}$
increasing to infinity, such that

\begin{equation}
\lim_{k\rightarrow \infty }\frac{Q(2^{2n_{k}|B|})}{2^{2n_{k}|B|}n_{k}^{|B|-1}%
}=0,\quad \frac{Q(2^{2n_{k}|B|})}{2^{|B|\left( 2n_{k}+1\right) }}\geq
1,\quad \forall k.  \label{cond1}
\end{equation}

From (\ref{est1}) we have 
\begin{eqnarray*}
&&\mes\left\{ \mathbf{x}\in \mathbb{I}^{d}:\left\vert L_{p_{n_{k}}\left(
B\right) }\circ R_{p_{n_{k}}\left( B^{\prime }\right) }\left(
\bigotimes\limits_{i\in B}D_{2^{2n_{k}+1}};\mathbf{x}\right) \right\vert
\gtrsim 2^{n_{k}\left( 2|B|-1\right) }\right\} \\
&\gtrsim &\frac{n_{k}^{|B|-1}}{2^{n_{k}\left( 2|B|-1\right) }}.
\end{eqnarray*}%
Then by the virtue of Theorem \ref{Garsia} there exists $E_{1},...,E_{r}\in 
\mathcal{E}$ and $\varepsilon _{1},...,\varepsilon _{r}=\pm 1$ such that 
\begin{eqnarray}
\mes\{\mathbf{x} &\in &\mathbb{I}^{d}:\left\vert
\sum\limits_{i=1}^{r_k}\varepsilon _{i}L_{p_{n_{k}}\left( B\right) }\circ
R_{p_{n_{k}}\left( B^{\prime }\right) }\left( \bigotimes\limits_{j\in
B}D_{2^{2n_{k}+1}};E_{i}\mathbf{x}\right) \right\vert  \label{low} \\
&>&2^{n_{k}\left( 2|B|-1\right) }\}>\frac{1}{8},  \notag
\end{eqnarray}%
where 
\begin{equation*}
r_k\sim \frac{2^{n_{k}\left( 2|B|-1\right) }}{n_{k}^{|B|-1}}.
\end{equation*}

Denote 
\begin{equation*}
\nu _{k}=\frac{2^{n_{k}\left( 4|B|-1\right) -1}}{r_{k}Q\left(
2^{2n_{k}|B|}\right) }
\end{equation*}%
and 
\begin{equation*}
\xi _{k}\left( \mathbf{x}\right) =\frac{2^{2|B|n_{k}-1}}{Q\left(
2^{2n_{k}|B|}\right) }M_{k}\left( \mathbf{x}\right) ,
\end{equation*}%
where 
\begin{equation*}
M_{k}\left( \mathbf{x}\right) =\frac{1}{r_{k}}\sum\limits_{i=1}^{r_{k}}%
\varepsilon _{i}\prod\limits_{j\in B}D_{2^{2n_{k}+1}}\left( E_{i}^{\left(
j\right) }x_{j}\right) ,
\end{equation*}%
\begin{equation*}
E_{i}:=\left( E_{i}^{\left( 1\right) },...,E_{i}^{\left( d\right) }\right)
\end{equation*}

Thus, from (\ref{low}) we obtain (\ref{est0}).

Finally, we prove that $\xi _{k}\in S_{Q}\left( 0,1\right) $. Since%
\begin{equation*}
\left\Vert M_{k}\right\Vert _{\infty }\leq 2^{|B|\left( 2n_{k}+1\right) },
\end{equation*}%
\begin{equation*}
\left\Vert M_{k}\right\Vert _{L_{1}\left( \mathbb{I}^{d}\right) }\leq 1,
\end{equation*}%
\begin{equation*}
\left\Vert \xi _{k}\right\Vert _{L_{Q}\left( \mathbb{I}^{d}\right) }\leq 
\frac{1}{2}\left[ \int\limits_{\mathbb{I}^{d}}Q\left( 2\left\vert \xi
_{k}\right\vert \right) +1\right] ,
\end{equation*}%
and%
\begin{equation*}
\frac{Q\left( u\right) }{u}<\frac{Q\left( u^{\prime }\right) }{u^{\prime }}%
,\left( 0<u<u^{\prime }\right)
\end{equation*}%
we can write

\begin{eqnarray*}
\Vert \xi _{k}\Vert _{L_{Q}(\mathbb{I}^{d})} &\leq &\frac{1}{2}\left[
1+\int\limits_{\mathbb{I}^{d}}Q\left( \frac{2^{2|B|n_{k}}\left\vert
M_{k}\left( x\right) \right\vert }{Q(2^{2|B|n_{k}})}\right) d\mathbf{x}%
\right] \\
&\leq &\frac{1}{2}\left[ 1+\int\limits_{\mathbb{I}^{d}}\frac{Q\left( \frac{%
2^{2|B|n_{k}}2^{|B|\left( 2n_{k}+1\right) }}{Q(2^{2|B|n_{k}})}\right) }{%
\frac{2^{2|B|n_{k}}2^{|B|\left( 2n_{k}+1\right) }}{Q(2^{2|B|n_{k}})}}\frac{%
2^{2|B|n_{k}}\left\vert M_{k}\left( \mathbf{x}\right) \right\vert }{%
Q(2^{2|B|n_{k}})}d\mathbf{x}\right] \\
&\leq &\frac{1}{2}\left[ 1+\int\limits_{\mathbb{I}^{d}}\frac{Q\left(
2^{2|B|n_{k}}\right) }{2^{2|B|n_{k}}}\frac{2^{2|B|n_{k}}\left\vert
M_{k}\left( \mathbf{x}\right) \right\vert }{Q(2^{2|B|n_{k}})}d\mathbf{x}%
\right] \\
&\leq &1.
\end{eqnarray*}

Hence, $\xi _{k}\in S_{Q}\left( 0,1\right) $, and Theorem \ref{baire} is
proved.
\end{proof}

The validity of Corollary \ref{measuredivergence} follows immediately from
Theorem \ref{baire} and Lemma \ref{GGT2}.

\end{document}